\documentclass[article]{memoir}

\usepackage{amsmath,amsthm,amssymb,proof,ifthen,marginnote,mathpazo}
\usepackage[usenames,svgnames]{xcolor}
\usepackage[pdftex,colorlinks,urlcolor=MidnightBlue,linkcolor=DarkGreen]{hyperref}

\theoremstyle{plain}
\newtheorem{thm}{Theorem}[chapter]
\newtheorem{lem}[thm]{Lemma}
\newtheorem{prop}[thm]{Proposition}

\newtheorem{cor}[thm]{Corollary}
\theoremstyle{definition}
\newtheorem{defn}[thm]{Definition}
\newtheorem{ex}{Exercise}
\theoremstyle{remark}

\def\val#1#2#3#4{\mathrm{val}_{#1,#2,#3}(#4)}
\def\sat#1#2#3#4{#1,#2,#3 \models #4}
\def\M{\mathfrak{M}}
\let\lif\rightarrow
\let\liff\leftrightarrow
\let\seq\equiv
\def\eps{\ensuremath{\varepsilon}}
\def\mrm#1{\ensuremath{\mathrm{#1}}}
\def\card#1{\left|#1\right|}
\def\deg#1{\mathrm{deg}(#1)}
\def\rk#1{\mathrm{rk}(#1)}

\def\Le{L_{\eps}}
\def\Lea{L_{\eps\forall}}
\def\La{L_\forall}
\def\Trm{\mrm{Trm}}
\def\Frm{\mrm{Frm}}
\def\Var{\mrm{Var}}
\def\Pred{\mrm{Pred}}
\def\Fct{\mrm{Fct}}
\def\Typ{\mrm{Typ}}
\def\FV{\mrm{FV}}
\def\meps#1#2{\ensuremath{\varepsilon_#1\,#2}}
\def\st#1#2#3{\ensuremath{#1[#2/#3]}}
\def\ST#1#2#3{\ensuremath{#1\{#2/#3\}}}
\def\ext{\mathrm{ext}}
\def\inst#1{#1^\textrm{inst}}

\def\EC{\mrm{EC}}
\def\PC{\ensuremath{\mathrm{EC}_\forall}}
\def\ECe{\ensuremath{\mathrm{EC}_\eps}}
\def\ECex{\ensuremath{\mathrm{EC}_\eps^\ext}}
\def\PCe{\ensuremath{\mathrm{EC}_{\eps\forall}}}

\newcommand*{\proves}[2][{}]{\mathrel{\vdash^{#1}_{#2}}}
\newcommand*{\nproves}[1]{\mathrel{\not\vdash_{#1}}}

\newcounter{version}
\date{November 13, 2014---v.~\theversion}

\newcommand{\revprint}[2]{\ifthenelse{#1 =
    \value{version}}{\marginnote{\color{green}\normalfont\normalsize{#2}}}{}}
\def\rev#1{\revprint{#1}{\textsc{rev}}}
\def\new#1{\revprint{#1}{\textsc{new}}}

\title{Lectures on the Epsilon Calculus}
\author{\href{http://ucalgary.ca/rzach}{Richard Zach}}

\begin{document}
\maketitle

\tableofcontents*

\chapter*{Introduction}

These notes were written for use in a short course on the epsilon
calculus I taught at the Vienna University of Technology as an Erasmus
Mundus Fellow in 2009.  I had at one point hoped to expand them
substantially, and this may yet happen.  In the meantime, perhaps they
prove useful to someone even if the results presented remain quite
basic, proofs are left as exercises, explanations are sparse (to say
the least), and there are no references (but see
\href{http://plato.stanford.edu/entries/epsilon-calculus/}{here}). Semantics
for the epsilon calculus, completeness proofs, as well as proofs of
the first epsilon theorem, are, after all, still hard to find, at
least in English.  I claim no credit to the results contained herein;
in particular, I've learned the intensional and extensional semantics
and completness results from Grigori Mints' lectures on the epsilon
calculus. Please \href{http://ucalgary.ca/rzach}{contact me} if you
find a mistake.

\begin{ex}
Find the mistakes in these notes.
\end{ex}

\chapter{Syntax}

\section{Languages}

\begin{defn}
  The language of of the elementary calculus~$L_\EC$ contains the
  following symbols:
\begin{enumerate}
\item Variables~\Var: $x_0$, $x_1$, \dots
\item Function symbols~$\Fct^n$ of arity $n$, for each $n \ge 0$: $f_0^n$,
  $f_1^n$, \dots
\item Predicate symbols~$\Pred^n$ of arity $n$, for each $n \ge 0$: $P^n_0$,
  $P^n_1$, \dots
\item Identity: $=$
\item Propositional constants: $\bot$, $\top$.
\item Propositional operators: $\lnot$, $\land$,
  $\lor$, $\lif$, $\liff$.
\item Punctuation: parentheses: $($, $)$; comma: , 
\end{enumerate}
For any language $L$, we denote by $L^{-}$ the language~$L$ without
the identity symbol, by $L_\eps$ the language $L$ plus the
symbol~$\eps$, and by $L_\forall$ the language~$L$ plus the
quantifiers $\forall$ and $\exists$.  We will usually leave out the
subscript~$\EC$, and write $\La$ for the language of the predicate
calculus, $\Le$ for the language of the \eps-calculus, and $\Lea$ for
the language of the extended epsilon calculus.
\end{defn}

\begin{defn}
  The \emph{terms}~$\Trm$ and \emph{formulas}~$\Frm$
  of~$\Lea$ are defined as follows.
\begin{enumerate}
\item Every variable~$x$ is a term, and $x$ is free in it.
\item If $t_1$, \dots, $t_n$ are terms, then $f^n_i(t_1, \dots, t_n)$
  is a term, and $x$ occurs free in it wherever it occurs free in
  $t_1$, \dots, $t_n$.
\item If $t_1$, \dots, $t_n$ are terms, then $P^n_i(t_1, \dots, t_n)$
  is an (atomic) formula, and $x$ occurs free in it wherever it occurs
  free in $t_1$, \dots, $t_n$.
\item $\bot$ and $\top$ are formulas.
\item If $A$ is a formula, then $\lnot A$ is a formula, with the same
  free occurrences of variables as~$A$.
\item If $A$ and $B$ are formulas, then $(A \land B)$, $(A \lor B)$,
  $(A \lif B)$, $(A \liff B)$ are formulas, with the same free
  occurrences of variables as $A$ and $B$.
\item If $A$ is a formula in which $x$ has a free occurrence but no
  bound occurrence, then $\forall x\, A$ and $\exists x\, A$ are
  formulas, and all occurrences of $x$ in them are bound.
\item If $A$ is a formula in which $x$ has a free occurrence but no
  bound occurrence, then $\meps x A$ is a term, and all occurrences of
  $x$ in it are bound.
\end{enumerate}
The terms~$\Trm(L)$ and formulas~$\Frm(L)$ of a langauge $L$ are those
terms and formulas of~$\Lea$ in the vocabulary of~$L$.
\end{defn}

If $E$ is an expression (term or formula), then $\FV(E)$ is the set of
variables which have free occurrences in~$E$.  $E$ is called
\emph{closed} if $\FV(E) = \emptyset$.  A closed formula is also
called a \emph{sentence}.

When $E$, $E'$ are expressions (terms or formulas), we write $E \seq
E'$ iff $E$ and $E'$ are syntactically identical up to a renaming of
bound variables.  We say that a term $t$ is \emph{free for $x$ in $E$} iff
$x$ does not occur free in the scope of an \eps-operator $\meps y{}$ or
quantifier $\forall y$, $\exists y$ for any~$y \in \FV(t)$.

If $E$ is an expression and $t$ is a term, we write $\st E x t$ for
the result of substituting every free occurrence of~$x$ in~$E$ by~$t$,
provided $t$ is free for $x$ in $E$, and renaming bound variables
in~$t$ if necessary.  

If $t$ is not free for $x$ in $E$, $\st E x t$ is any formula $\st {E'}
x t$ where $E' \seq E$ and $t$ is free for $x$ in $E'$.  If $E' \seq
\st{\st E {x_1} {t_1} \dots}{x_n}{t_n}$, $E'$ it is called an
\emph{instance of~$E$}.

We write $E(x)$ to indicate that $x \in \FV(E)$, and $E(t)$ for $\st E
x t$. It will be apparent from the context which variable~$x$ is
substituted for.

\begin{defn}\new{2}
  A term~$t$ is a \emph{subterm} of an expression (term or
  formula)~$E$, if for some $E'(x)$, $E \seq \st {E'(x)} x t$.  It is
  a \emph{proper subterm} of a term $u$ if it is a subterm of $u$ but
  $t \not\seq u$.

  A term~$t$ is an \emph{immediate subterm} of an expression~$E$ if
  $t$ is a subterm of $E$, but not a subterm of a proper subterm
  of~$E$.
\end{defn}

\begin{defn}\new{2}
  If $t$ is a subterm of $E$, i.e., for some $E'$ we have $E \seq \st
  {E'}{x}{t}$, then $\ST E t u$ is $\st {E'} x u$.
\end{defn}

We intend $\ST E t u$ to be the result of replacing every occurrence
of $t$ in $E$ by $u$. But, the ``brute-force'' replacement of every
occurrence of $t$ in $u$ may not be what we have in mind here. (a)~We
want to replace not just every occurrence of $t$ by $u$, but every
occurrence of a term $t' \seq t$.  (b)~$t$ may have an occurrence in
$E$ where a variable in $t$ is bound by a quantifier or \eps{} outside
$t$, and such occurrences shouldn't be replaced (they are not subterm
occurrences). (c)~When replacing $t$ by $u$, bound variables in $u$
might have to be renamed to avoid conflicts with the bound variables
in $E'$ and bound variables in $E'$ might have to be renamed to avoid
free variables in $u$ being bound.

\begin{defn}[\eps-Translation]
  If $E$ is an expression, define $E^\eps$ by:
  \begin{enumerate}
  \item $E^\eps = E$ if $E$ is a variable, a constant symbol,
    or~$\bot$.
  \item If $E = f^n_i(t_1, \dots, t_n)$, $E^\eps = f^n_i(t_1^\eps,
    \dots, t_n^\eps)$.
  \item If $E = P^n_i(t_1, \dots, t_n)$, $E^\eps = P^n_i(t_1^\eps,
    \dots, t_n^\eps)$.
  \item If $E = \lnot A$, then $E^\eps = \lnot A^\eps$.
  \item If $E = (A \land B)$, $(A \lor B)$, $(A \lif B)$, or $(A \liff
    B)$, then $E^\eps = (A^\eps \land B^\eps)$, $(A^\eps \lor
    B^\eps)$, $(A^\eps \lif B^\eps)$, or $(A^\eps \liff B^\eps)$,
    respectively.
  \item \rev{2} If $E = \exists x\, A(x)$ or $\forall x\, A(x)$, then $E^\eps
    = A^\eps(\meps x{A(x)^\eps})$ or $A^\eps(\meps x {\lnot A(x)^\eps})$.
  \item If $E = \meps x {A(x)}$, then $E^\eps = \meps x{A(x)^\eps}$.
  \end{enumerate}
\end{defn}

\section[\eps-Types, Degree, and Rank]{\new{2}\eps-Types, Degree, and
  Rank}

\begin{defn}
  An \eps-term $p \seq \meps x {B(x; x_1, \dots, x_n)}$ is a
  \emph{type of an \eps-term~$\meps x {A(x)}$} iff
  \begin{enumerate}
  \item $p \seq \st{\st{\meps x {A(x)}}{x_1}{t_1}\dots}{x_n}{t_n}$
  for some terms $t_1$, \dots,~$t_n$.
  \item $\FV(p) = \{x_1, \dots, x_n\}$.
  \item $x_1$, \dots, $x_n$ are all immediate subterms of $p$.
  \item Each $x_i$ has exactly one occurrence in~$p$.
  \item The occurrence of $x_i$ is left of the occurrence of $x_j$ in
    $p$ if $i < j$.
  \end{enumerate}
  We denote the set of types of a langauge as~\Typ.
\end{defn}

\begin{prop}
  The type of an epsilon term~$\meps x {A(x)}$ is unique up to
  renaming of bound, and disjoint renaming of free variables, i.e., if
  $p = \meps x{B(x; x_1, \dots, x_n)}$, $p' = \meps y{B'(y; y_1, \dots,
      y_m)}$ are types of $\meps x {A(x)}$, then $n = m$ and $p' \seq
      \st{\st {p}{x_1}{y_1}\dots}{x_n}{y_n}$
\end{prop}

\begin{proof} Exercise.
\end{proof}

\begin{defn}
  An \eps-term $e$ is \emph{nested in} an \eps-term $e'$ if $e$ is a
  proper subterm of~$e$.
\end{defn}

\begin{defn}
  The \emph{degree~$\deg e$} of an \eps-term~$e$ is defined as follows:
  \begin{enumerate}
  \item $\deg e = 1$ iff $e$ contains no nested \eps-terms.
  \item $\deg e = \max\{\deg {e_1}, \dots, \deg{e_n}\} + 1$ if $e_1$,
    \dots,~$e_n$ are all the \eps-terms nested in~$e$.
  \end{enumerate}
  For convenience, let $\deg{t} = 0$ if $t$ is not an \eps-term.
\end{defn}

\begin{defn}
  An \eps-term $e$ is \emph{subordinate to} an \eps-term $e' = \meps x
  A(x)$ if some $e'' \seq e$ occurs in $e'$ and $x \in \FV(e'')$.
\end{defn}

Note that if $e$ is subordinate to $e'$ it is \emph{not} a subterm of
$e'$, because $x$ is free in $e$ and so the occurrence of $e$ (really,
of the variant $e''$) in $e'$ is in the scope of~$\eps_x$.  One might
think that replacing $e$ in $\meps x{A(x)}$ by a new variable $y$
would result in an \eps-term $\meps x{A'(y)}$ so that $e' \equiv \st
{\meps x{A'(y)}}{y}{e}$. But (a) $\meps x{A'(y)}$ is not in general a
term, since it is not guaranteed that $x$ is free in $A'(y)$ and (b)
$e$ is not free for $y$ in $\meps x {A'(y)}$.

\begin{defn}
  The \emph{rank~$\rk e$} of an \eps-term~$e$ is defined as follows:
  \begin{enumerate}
  \item $\rk e = 1$ iff $e$ contains no subordinate \eps-terms.
  \item $\rk e = \max\{\rk {e_1}, \dots, \rk{e_n}\} + 1$ if $e_1$,
    \dots,~$e_n$ are all the \eps-terms subordinate to~$e$.
  \end{enumerate}
\end{defn}

\begin{prop}
  If $p$ is the type of $e$, then $\rk{p} = \rk{e}$.
\end{prop}

\begin{proof} Exercise.
\end{proof}

\section{Axioms and Proofs}

\begin{defn} The axioms of the \emph{elementary calculus}~\EC{} are
\begin{align}
& A & \text{for any tautology~$A$} \tag{Taut} \\
t & = t & \text{for any term~$t$} \tag{$=_1$} \\
t = u & \lif (\st A x t \liff \st A x u) \tag{$=_2$}
\end{align}
and its only rule of inference is
\[\infer[$MP$]{A}{A & A \lif B}\]
The axioms and rules of the (intensional) \emph{\eps-calculus}~\ECe{}
are those of~\EC{} plus the \emph{critical formulas}
\begin{align}
A(t) \lif A(\meps x {A(x)}). \tag{crit}
\end{align}
The axioms and rules of the \emph{extensional \eps-calculus}~$\ECe^\ext$
are those of~\ECe{} plus
\begin{align}
  (\forall x(A(x) \liff B(x)))^\eps & \lif \meps x {A(x)} = \meps x
  {B(x)} \tag{ext} \\
  & \text{that is,}\notag\\
  A(\meps x{\lnot(A(x) \liff B(x))}) \liff B(\meps x{\lnot(A(x) \liff
    B(x))}) & \lif \meps x {A(x)} = \meps x {B(x)} \notag
\end{align}
The axioms and rules of \PC, \PCe, $\PCe^\ext$ are those of \EC, \ECe,
$\ECe^\ext$, respectively, together with the axioms
\begin{align}
A(t) & \lif \exists x\, A(x)  \tag{Ax$\exists$} \\
\forall x\, A(x) & \lif A(t) \tag{Ax$\forall$}
\end{align}
and the rules
\[
\infer[R\exists]{\exists x\, A(x) \lif B}{A(x) \lif B} \qquad
\infer[R\forall]{B \lif \forall x\, A(x)}{B \lif A(x)} 
\]
Applications of these rules must satisfy the \emph{eigenvariable
  condition}, viz., the variable $x$ must not appear in the conclusion
or anywhere below it in the proof.
\end{defn}

\begin{defn}
  If $\Gamma$ is a set of formulas, a \emph{proof of $A$ from $\Gamma$
    in $\PCe^\ext$} is a sequence~$\pi$ of formulas $A_1$, \dots, $A_n = A$
  where for each $i \le n$, one of the following holds:
  \begin{enumerate}
  \item $A_i \in \Gamma$.
  \item $A_i$ is an instance of an axiom.
  \item $A_i$ follows from some $A_k$, $A_l$ ($k$, $l < i$) by (MP),
    i.e., $A_i \seq C$, $A_k \seq B$, and $A_l \seq B \lif C$.
  \item $A_i$ follows from some $A_j$ ($j < i$) by (R$\exists$), i.e.,
    i.e., $A_i \equiv \exists x\,B(x) \lif C$, $A_j \equiv B(x) \lif
    C$, and $x$ is an eigenvariable, i.e., it satisfies $x \notin
    \FV(A_k)$ for any $k \ge i$ (this includes $k = i$, so $x \notin
    \FV(C))$.
  \item $A_i$ follows from some $A_j$ ($j < i$) by (R$\forall$), i.e.,
    i.e., $A_i \equiv C \lif \forall x\,B(x)$, $A_j \equiv C \lif
    B(x)$, and the eigenvariable condition is satisfied.
  \end{enumerate}
  If $\pi$ only uses the axioms and rules of \EC, \ECe, $\ECe^\ext$,
  etc., then it is a proof of $A$ from $\Gamma$ in \EC, \ECe,
  $\ECe^\ext$, etc., and we write $\Gamma \proves[\pi]{} A$, $\Gamma
  \proves[\pi]{\eps} A$, $\Gamma \proves[\pi]{\eps\ext} A$, etc. 

  We say that $A$ is provable from $\Gamma$ in \EC, etc. ($\Gamma
  \proves{} A$, etc.), if there is a proof of $A$ from $\Gamma$ in
  \EC, etc.
\end{defn}

Note that our definition of proof, because of its use of $\seq$,
includes a tacit rule for renaming bound variables.  Note also that
substitution into members of~$\Gamma$ is \emph{not} permitted.
However, we can simulate a provability relation in which substitution
into members of $\Gamma$ is allowed by considering $\inst \Gamma$, the
set of all substitution instances of members of~$\Gamma$.  If $\Gamma$
is a set of sentences, then $\inst\Gamma = \Gamma$.

\begin{prop}\label{proof-subst}
  If $\pi = A_1$, \dots, $A_n \equiv A$ is a proof of $A$ from
  $\Gamma$ and $x \notin \FV(\Gamma)$ is not an eigenvariable in
  $\pi$, then $\st \pi x t = \st {A_1} x t$, \dots, $\st {A_n} x t$ is
  a proof of $\st A x t$ from $\inst \Gamma$.

  In particular, if $\Gamma$ is a set of sentences and $\pi$ is a
  proof in \EC, \ECe, or \ECex, then $\st \pi x t$ is a proof of $\st
  A x t$ from $\Gamma$ in \EC, \ECe, or \ECex
\end{prop}

\begin{proof}
Exercise.
\end{proof}

\begin{lem}\label{ded-lemma}
  If $\pi$ is a proof of $B$ from $\Gamma \cup \{A\}$, then there is a
  proof $\pi[A]$ of $A \lif B$ from $\Gamma$, provided $A$ 
  contains no eigenvariables of~$\pi$ free.
\end{lem}

\begin{proof}
Construct $\pi[A]_0 = \emptyset$. Let $\pi_{i+1}[A] = \pi_{i}[A]$ plus
additional formulas, depending on $A_i$:
\begin{enumerate}
\item If $A_i \in \Gamma$, add $A \lif A$, if $A_i \equiv A$, or else
  add $A_i$, the tautology $A_i \lif (A \lif A_i)$, and $A \lif A_i$.
  The last formula follows from the previous two by (MP).
\item If $A_i$ is a tautology, add $A \lif A_i$, which is also a tautology.
\item If $A_i$ follows from $A_k$ and $A_l$ by (MP), i.e., $A_i \equiv
  C$, $A_k \equiv B$ and $A_l \equiv B \lif C$, then $\pi[A]_i$
  contains $A \lif B$ and $A \lif (B \lif C)$.  Add the tautology $(A
  \lif B) \lif ((B \lif C) \lif (A \lif C)$ and $A \lif C$. The latter
  follows from the former by two applications of (MP).
\item If $A_i$ follows from $A_j$ by (R$\exists$), i.e., $A_i \equiv
  \exists x\, B(x) \lif C$ and $A_j \equiv B(x) \lif C$, then
  $\pi[A]_i$ contains $A \lif (B(x) \lif C)$. $\pi[A]_{i+1}$ is 
\begin{align*}
  & \pi[A]_i \\
  & (A \lif (B(x) \lif C)) \lif (B(x) \lif (A \lif C)) & \text{(taut)} \\
  & B(x) \lif (A \lif C) & \text{(MP)} \\
  & \exists x\, B(x) \lif (A \lif C) & \text{(R$\exists$)} \\
  & (\exists x\, B(x) \lif (A \lif C)) \lif
  (A \lif (\exists x\, B(x) \lif C)) & \text{(taut)} \\
  & A \lif (\exists x\, B(x) \lif C) & \text{(MP)}
\end{align*}
Since $x \notin \FV(A)$, the eigenvariable condition is satisfied.
\item Exercise: $A_i$ follows by (R$\forall$).
\end{enumerate}
Now take $\pi[A] = \pi[A]_i$.
\end{proof}

\begin{thm}[Deduction Theorem]\label{deduction-thm}
  If $\Sigma \cup \{A\}$ is a set of sentences, $\Sigma \proves{} A \lif B$
  iff $\Sigma \cup \{A\} \proves{} B$.
\end{thm}

\begin{cor}\label{incons}
  If $\Sigma \cup \{A\}$ is a set of sentences, $\Sigma \proves{} A$
  iff $\Sigma \cup \{\lnot A\} \proves{} \bot$.
\end{cor}

\begin{lem}[\eps-Embedding Lemma]
  If $\Gamma \proves[\pi]{\eps\forall} A$, then there is a proof
  $\pi^\eps$ so that $\inst{{\Gamma^\eps}} \proves[\pi^\eps]{\eps}
  A^\eps$
\end{lem}

\begin{proof}
Exercise.
\end{proof}

\chapter{Semantics}

\section{Semantics for $\PCe^\ext$}

\begin{defn}
  A \emph{structure}~$\M = \langle \card \M, (\cdot)^\M\rangle$
  consists of a nonempty \emph{domain}~$\card \M \neq \emptyset$ and a
  maping $(\cdot)^\M$ on function and predicate symbols where:
\begin{align*}
(f^0_i)^\M & \in \card \M \\
(f^n_i)^M & \in \M^{\M^n} \\
(P^n_i)^\M & \subseteq \M^n
\end{align*}
\end{defn}

\begin{defn}
  An \emph{extensional choice function~$\Phi$ on $\M$} is a function
  $\Phi\colon \wp(\card\M) \to \card\M$ where $\Phi(X) \in X$ whenever
  $X \neq \emptyset$.
\end{defn}

Note that $\Phi$ is total on $\wp(\card\M)$, and so
$\Phi(\emptyset) \in \card{\M}$.

\begin{defn}
  An \emph{assignment~$s$ on $\M$} is a function $s\colon \Var \to
  \card\M$.

  If $x \in \Var$ and $m \in \card\M$, $\st s x m$ is the assignment
  defined by
\[
\st s x m(y) = \begin{cases} m & \text{if $y = x$} \\ s(y) &
  \text{otherwise}
\end{cases}
\] 
\end{defn}

\begin{defn}\label{ext-sat}
  The \emph{value~$\val \M \Phi s t$ of a term} and the
  \emph{satisfaction relation $\sat \M \Phi s A$} are defined as
  follows:
\begin{enumerate}
\item $\val \M \Phi s x =  s(x)$
\item $\sat \M \Phi s \top$ and $\M, \Phi, s \not\models \bot$
\item $\val \M \Phi s {f^n_i(t_1, \dots, t_n)} = (f^n_i)^\M(\val \M
  \Phi s {t_1}, \dots, \val \M \Phi s {t_n})$ 
\item $\sat \M \Phi s {P^n_i(t_1, \dots, t_n)}$ iff $\langle\val \M
  \Phi s {t_1}, \dots, \val \M \Phi s {t_n}\rangle \in (P^n_i)^\M$
\item\label{epsilon-sat} $\val \M \Phi s {\meps x{A(x)}} = \Phi(\val \M
  \Phi s {A(x)})$ where
\[
\val \M \Phi s {A(x)} = \{ m \in \card\M : \sat \M \Phi {\st s x m} A(x)\}
\]
\item $\sat \M \Phi s {\exists x\, A(x)}$ iff for some $m \in
  \card\M$, $\sat \M \Phi {\st s x m} {A(x)}$
\item $\sat \M \Phi s {\forall x\, A(x)}$ iff for all $m \in
  \card\M$, $\sat \M \Phi {\st s x m} {A(x)}$
\end{enumerate}
\end{defn}

\begin{prop}
  If $s(x) = s'(x)$ for all $x \notin \FV(t) \cup \FV(A)$, then $\val
  \M \Phi s t = \val \M \Phi {s'} t$ and $\sat \M \Phi s A$ iff $\sat
  \M \Phi {s'} A$.
\end{prop}

\begin{proof} Exercise. \end{proof}

\begin{prop}[Substitution Lemma]
If $m = \val \M \Phi s u$, then 
  $\val \M \Phi s {t(u)} = \val \M \Phi {\st s x m} {t(x)}$ and
  $\sat \M \Phi s {A(u)}$ iff $\sat \M \Phi {\st s x m} {A(x)}$
\end{prop}

\begin{proof} Exercise. \end{proof}

\begin{defn}
\begin{enumerate}
\item $A$ is \emph{locally true} in $\M$ with respect to $\Phi$ and
  $s$ iff $\sat \M \Phi s A$.
\item $A$ is \emph{true} in $\M$ with respect to $\Phi$, $\M, \Phi
  \models A$, iff for all $s$ on $\M$: $\sat \M \Phi s A$.
\item $A$ is \emph{generically true} in $\M$ with respect to~$s$, $\M,
  s \models^g A$, iff for all choice functions $\Phi$
  on~$\M$: $\sat \M \Phi s A$.
\item $A$ is \emph{generically valid} in $\M$, $\M \models A$, if for
  all choice functions $\Phi$ and assignments~$s$ on~$\M$: $\sat \M
  \Phi s A$.
\end{enumerate}
\end{defn}

\begin{defn} Let $\Gamma \cup\{A\}$ be a set of formulas.
\begin{enumerate}
\item $A$ is a \emph{local consequence} of $\Gamma$, $\Gamma
  \models^l A$, iff for all $\M$, $\Phi$, and $s$:\\
  \qquad if $\sat \M \Phi s \Gamma$ then $\sat \M \Phi s A$.
\item $A$ is a \emph{truth consequence} of $\Gamma$, $\Gamma
  \models A$, iff for all $\M$, $\Phi$:\\
  \qquad if $\M, \Phi \models \Gamma$
  then $\M, \Phi \models A$.
\item $A$ is a \emph{generic consequence} of $\Gamma$, $\Gamma
  \models^g A$, iff for all $\M$ and $s$:\\
  \qquad if $\M, s \models^g \Gamma$
  then $\M \models A$.
\item $A$ is a \emph{generic validity consequence} of $\Gamma$, $\Gamma
  \models^v A$, iff for all $\M$:\\
  \qquad if $\M \models^v \Gamma$ then $\M
  \models A$.
\end{enumerate}
\end{defn}

\begin{ex}
  What is the relationship between these consequence relations?  For
  instance, if $\Gamma \models^l A$ then $\Gamma \models A$ and
  $\Gamma \models^g A$, and if eiter $\Gamma \models A$ or $\Gamma
  \models^g A$, then $\Gamma \models^v A$.  Are these containments
  strict?  Are they identities (in general, and in cases where the
  language of $\Gamma$, $A$ is restricted, or if $\Gamma$, $A$ are
  sentences)?  For instance:
\end{ex}

\begin{prop}
  If $\Sigma \cup \{A\}$ is a set of sentences, $\Sigma \models^l A$
  iff $\Sigma \models A$
\end{prop}

\begin{prop}
  If $\Sigma \cup \{A, B\}$ is a set of sentences, $\Sigma \cup \{A\}
  \models B$ iff $\Sigma \models A \lif B$.
\end{prop}

\begin{proof} Exercise. \end{proof}

\begin{cor}
  If $\Sigma \cup \{A\}$ is a set of sentences, $\Sigma \models A$
  iff for no $\M$, $\Phi$, $\M \models \Sigma \cup \{\lnot A\}$
\end{cor}

\begin{proof} Exercise. \end{proof}

\begin{ex}
For which of the other consequence relations, if any, do these results hold?
\end{ex}

\section{Soundness for $\PCe^\ext$}

\begin{thm}
If $\Gamma \proves{\eps\forall} A$, then $\Gamma \models^l A$.
\end{thm}

\begin{proof}
  Suppose $\sat\Gamma\Phi s \Gamma$.  We show by induction on the
  length $n$ of a proof $\pi$ that $\sat \M \Phi s' A$ for all $s'$
  which agree with $s$ on $\FV(\Gamma)$.  We may assume that no
  eigenvariable~$x$ of $\pi$ is in $\FV(\Gamma)$ (if it is, let $y
  \notin \FV(\pi)$ and not occurring in~$\pi$; consider $\st \pi x y$
  instead of $pi$).

  If $n = 0$ there's nothing to prove. Otherwise, we distinguish cases
  according to the last line $A_n$ in $\pi$:
\begin{enumerate}
\item $A_n \in \Gamma$. The claim holds by assumption.
\item $A_n$ is a tautology. Obvious.
\item $A_n$ is an identity axiom. Obvious.
\item $A_n$ is a critical formula, i.e., $A_n \equiv A(t) \lif A(\meps
  x {A(x)})$. Then either $\sat \M \Phi s {A(t)}$ or not (in which
  case there's nothing to prove). If yes, $\sat\M \Phi {s[x/m]} {A(x)}$
  for $m = \val \M \Phi s {t}$, and so $Y = \val \M \Phi s {A(x)} \neq
  \emptyset$.  Consequently, $\Phi(Y) \in Y$, and hence $\sat \M \Phi
  s {A(\meps x {A(x)})}$.
\item $A_n$ is an extensionality axiom. Exercise.
\item $A_n$ follows from $B$ and $B \lif C$ by (MP). By induction
  hypothesis, $\sat\M\Phi s B$ and $\sat\M\Phi s{B \lif C}$.
\item $A$ follows from $B(x) \lif C$ by (R$\exists$), and $x$
  satisfies the eigenvariable condition. Exercise.
\item $A$ follows from $C \lif B(x)$ by (R$\forall$), and $x$
  satisfies the eigenvariable condition. Exercise.
\end{enumerate}
\end{proof}

\begin{ex}
  Complete the missing cases.
\end{ex}

\section{Completeness for $\PCe^\ext$}

\begin{lem}
  If $\Gamma$ is a set of sentences in $\Le$ and $\Gamma
  \nproves{\eps} \bot$, then there are $\M$, $\Phi$ so that $\M,
  \Phi \models \Gamma$.
\end{lem}

\begin{thm}[Completeness]
  If $\Gamma \cup \{A\}$ are sentences in $\Le$ and $\Gamma \models
  A$, then $\Gamma \proves{\eps} A$.
\end{thm}

\begin{proof}
  Suppose $\Gamma \not\models A$. Then for some $\M$, $\Phi$ we have
  $\M, \Phi \models \Gamma$ but $\M, \Phi \not\models A$. Hence $\M,
  \Phi \models \Gamma \cup \{\lnot A\}$.  By the Lemma, $\Gamma \cup
  \{\lnot A\} \proves{\eps} \bot$. By Corollary~\ref{incons}, $\Gamma
  \proves{\eps} A$.
\end{proof}

The proof of the Lemma comes in several stages.  We have to show that
if $\Gamma$ is consistent, we can construct $\M$, $\Phi$, and $s$ so
that $\sat \M \Phi s \Gamma$. Since $\FV(\Gamma) = \emptyset$, we then
have $\M, \Phi \models \Gamma$.

\begin{lem}
  If $\Gamma \nproves{\eps} \bot$, there is $\Gamma^* \supseteq
  \Gamma$ with (1) $\Gamma^* \nproves{\eps} \bot$ and (2) for all
  formulas $A$, either $A \in \Gamma^*$ or $\lnot A \in \Gamma^*$.
\end{lem}

\begin{proof}
  Let $A_1$, $A_2$, \dots\ be an enumeration of $\Frm_\eps$. Define
  $\Gamma_0 = \Gamma$ and \[ \Gamma_{n+1} =
\begin{cases}
  \Gamma_n \cup \{A_n\} & 
  \text{if $\Gamma_n \cup \{A_n\} \nproves{\eps} \bot$} \\
  \Gamma_n \cup \{\lnot A_n\} & 
  \text{if $\Gamma_n \cup \{\lnot A_n\}
    \nproves{\eps} \bot$ otherwise}
\end{cases}
\]
Let $\Gamma^* = \bigcup_{n\ge 0} \Gamma_n$. Obviously, $\Gamma
\subseteq \Gamma^*$. For (1), observe that if $\Gamma^*
\proves[\pi]{\eps} \bot$, then $\pi$ contains only finitely many
formulas from~$\Gamma^*$, so for some $n$, $\Gamma_n
\proves[\pi]{\eps} \bot$. But $\Gamma_n$ is consistent by definition.

To verify (2), we have to show that for each $n$, either $\Gamma_n
\cup \{A_n\} \nproves{\eps} \bot$ or $\Gamma_n \cup \{\lnot A\}
\nproves{\eps} \bot$. For $n = 0$, this is the assumtion of the lemma.
So suppose the claim holds for $n-1$.  Suppose $\Gamma_n \cup \{A\}
\proves[\pi]{\eps} \bot$ and $\Gamma_n \cup \{\lnot A\}
\proves[\pi']{\eps} \bot$. Then by the Deduction Theorem, we have
$\Gamma_n \proves[{\pi[A]}] A \lif \bot$ and $\Gamma_n
\proves[{\pi'[A']}] \lnot A \lif \bot$. Since $(A \lif \bot) \lif
((\lnot A \lif \bot) \lif \bot)$ is a tautology, we have $\Gamma_n
\proves{\eps} \bot$, contradicting the induction hypothesis.
\end{proof}

\begin{lem}\label{lem-closed}
If $\Gamma^* \vdash{\eps} B$, then $B \in \Gamma^*$.
\end{lem}

\begin{proof}
  If not, then $\lnot B \in \Gamma^*$ by maximality, so $\Gamma^*$
  would be inconsistent.
\end{proof}

\begin{defn}
  Let $\approx$ be the relation on $\Trm_\eps$ defined by 
\[ 
t \approx u \text{ iff } t = u \in \Gamma^*
\]
It is easily seen that $\approx$ is an equivalence relation.  Let
$\widetilde t = \{u : u \approx t\}$ and $\widetilde \Trm =
\{\widetilde t : t \in \Trm\}$.
\end{defn}

\begin{defn}
A set $T \in \widetilde \Trm$ is \emph{represented by $A(x)$}
if $T = \{\widetilde t : A(t) \in \Gamma^*\}$.

Let $\Phi_0$ be a fixed choice function on $\widetilde \Trm$, and define
\[
\Phi(T) = \begin{cases}
\widetilde{\meps x {A(x)}} & \text{if $T$ is represented by $A(x)$}\\
\Phi_0(T) & \text{otherwise.}
\end{cases}
\]
\end{defn}

\begin{prop}
$\Phi$ is a well-defined choice function on $\widetilde \Trm$.
\end{prop}

\begin{proof}
  Exercise. Use (ext) for well-definedness and (crit) for choice
  function.
\end{proof}

Now let $\M =\langle \widetilde \Trm, (\cdot)^\M\rangle$ with $c^\M =
\widetilde c$, $(P_i^n)^\M = \{\langle\widetilde t_1, \dots,
\widetilde t_1\rangle : P_i^n(t_1, \ldots, t_n)\}$, and let $s(x) =
\widetilde s$.

\begin{prop}
  $\sat \M \Phi s {\Gamma^*}$.
\end{prop}

\begin{proof}
  We show that $\val \M \Phi s t = \widetilde t$ and $\sat \M \Phi s
  A$ iff $A \in \Gamma^*$ by simultaneuous induction on the complexity
  of $t$ and $A$.

  If $t = c$ is a constant, the claim holds by definition of
  $(\cdot)^\M$.  If $A = \bot$ or $= \top$, the claim holds by
  Lemma~\ref{lem-closed}.

  If $A \equiv P^n(t_1, \ldots, t_n)$, then by induction hypothesis,
  $\val \M \Phi s t_i = \widetilde {t_i}$.  By definition of
  $(\cdot)^\M$, $\langle \widetilde{t_1}, \dots,
  \widetilde{t_n}\rangle \in (P^n_i)(t_1, \dots, t_n)$ iff $P^n_i(t_1,
  \dots, t_n) \in \Gamma^*$.

  If $A\equiv \lnot B$, $(B \land C)$, $(B \lor C)$, $(B \lif C)$, $(B
  \liff C)$, the claim follows immediately from the induction
  hypothesis and the definition of $\models$ and the closure
  properties of $\Gamma^*$. For instance, $\sat \M \Phi s {(B \land
    C)}$ iff $\sat \M \Phi s B$ and $\sat \M \Phi s C$. By induction
  hypothesis, this is the case iff $B \in \Gamma^*$ and $C \in
  \Gamma^*$. But since $B, C \proves{\eps} B \land C$ and $B \land C
  \proves{\eps} B$ and $\proves{\eps} C$, this is the case iff $(B
  \land C) \in \Gamma^*$.  Remaining cases: Exercise.

  If $t \seq \meps x {A(x)}$, then $\val \M \Phi s t = \Phi(\val \M
  \Phi s {A(x)})$. Since $\val \M \Phi s {A(x)} $ is represented by
  $A(x)$ by induction hypothesis, we have $\val \M \Phi s t =
  \widetilde{\meps x {A(x)}}$ by definition of $\Phi$.
\end{proof}

\begin{ex}
Complete the proof. 
\end{ex}

\begin{ex}
Generalize the proof to $\Lea$ and $\PCe$.
\end{ex}

\begin{ex}
  Show $\ECe$ without ($=_1$) and ($=_2$), (ext), and the additional
  axiom
\begin{equation}
  (\forall x(A(x) \liff B(x)))^\eps \lif (C(\meps x {A(x)}) 
  \liff C(\meps x{B(x)})) \tag{$\ext^-$}
\end{equation}
is complete for $\models$ in the language $\Lea^-$.
\end{ex}

\section[Semantics for \PCe]{Semantics for \PCe\new{2}}

In order to give a complete semantics for \PCe, i.e., for the calculus
without the extensionality axion~(ext), it is necessary to chnage the
notion of choice function so that two \eps-terms \meps x {A(x)} and
\meps x {B(x)} may be assigned different representatives even when
$\sat \M \Phi s {\forall x(A(x) \liff B(x))}$, since then the negation
of (ext) is consistent in the resulting calculus.  The idea is to add
the \eps-term itself as an additional argument to the choice function.
However, in order for this semantics to be sound for the
calculus---specifically, in order for ($=_2$) to be valid---we have to
use not \eps-terms but \eps-types.

\begin{defn}
  An \emph{intensional choice operator} is a mapping $\Psi\colon \Typ
  \times \card{\M}^{<\omega} \to \card{\M}^{\wp(\card{\M})}$ such that
  for every type $p = \meps x{A(x; y_1, \dots, y_n)}$ is a type, and
  $m_1$, \dots,~$m_n \in \card{\M}$, $\Psi(p, m_1, \dots, m_n)$ is a
  choice function.
\end{defn}

\begin{defn}
  If $\M$ is a structure, $\Psi$ an intensional choice operator, and
  $s$ an assignment, $\val \M \Psi s t$ and $\sat \M \Psi s A$ is
  defined as before, except (5) in Definition~\ref{ext-sat} is
  replaced by:
  \begin{enumerate}
  \item[($\ref{epsilon-sat}'$)] $\val \M \Psi s {\meps x{A(x)}} =
    \Psi(p, m_1, \dots, m_n)(\val \M \Phi s {A(x)})$ where 
    \begin{enumerate}
    \item $p = \meps x{A'(x; x_1, \dots, x_n)}$ is the type of $\meps
      x {A(x)}$,
    \item $t_1$, \dots, $t_n$ are the subterms corresponding to $x_1$,
      \dots, $x_n$, i.e., $\meps x{A(x)} \seq \meps x{A'(x; t_1,
        \dots, t_n)}$, 
    \item $m_i = \val \M \Psi s t_1$, and 
    \item
    $\val \M \Phi s {A(x)} = 
    \{ m \in \card\M : \sat \M \Psi {\st s x m} A(x)\}$
    \end{enumerate}
\end{enumerate}
\end{defn}

\begin{ex}
  Prove the substitution lemma for this semantics.
\end{ex}

\begin{ex}
  Prove soundness.
\end{ex}

\begin{ex}
  Prove completeness of $\PCe$ for this semantics.
\end{ex}

\begin{ex}
  Define a semantics for the language without $=$ where the choice
  operatore takes $\eps$-terms as arguments. Is the semantics
  sound and complete for $\PCe^-$?
\end{ex}

\chapter[The First Epsilon Theorem]{The First Epsilon Theorem\new{2}}

\section{The Case Without Identity}

\begin{thm}
  If $E$ is a formula not containing any \eps-terms
  and $\proves{\PCe} E$, then $\proves{\EC} E$.
\end{thm}

\begin{defn}
  An \eps-term $e$ is \emph{critical in~$\pi$} if $A(t) \lif A(e)$ is
  one of the critical formulas in $\pi$.  The \emph{rank~$\rk{\pi}$ of
    a proof~$\pi$} is the maximal rank of its critical \eps-terms. The
  \emph{$r$-degree~$\deg{\pi, r}$} of $\pi$ is the maximum degree of
  its critical \eps-terms of rank~$r$. The \emph{$r$-order~$o(\pi,
    r)$} of $\pi$ is the number of different (up to renaming of bound
  variables) critical \eps-terms of rank~$r$.
\end{defn}

\begin{lem}
  If $e = \meps x{A(x)}$, $\meps y {B(y)}$ are critical in $\pi$,
  $\rk{e} = \rk{\pi}$, and $B^* \equiv B(u) \lif B(\meps y{B(y)})$ is a
  critical formula in~$\pi$. Then, if $e$ is a subterm of $B^*$, it is
  a subterm of $B(y)$ or a subterm of~$u$.
\end{lem}

\begin{proof}
  Suppose not. Then, since $e$ is a subterm of $B^*$, we have $B(y)
  \seq B'(\meps x{A'(x, y)}, y)$ and either $e \seq \meps x{A'(x, u)}$
  or $e \seq \meps x{A'(x, \meps y{B(y)})}$. In each case, we see that
  $\meps x{A'(x, y)}$ and $e$ have the same rank, since the latter is
    an instance of the former (and so have the same type). On the
    other hand, in either case, $\meps y {B(y)}$ would be
  \[
  \meps y {B'(\meps x{A'(x, y)}, y)}
  \]
  and so would have a higher rank than $\meps x{A'(x, y)}$ as that
  \eps-term is subordinate to it.  This contradicts $\rk{e} = \rk{\pi}$.
\end{proof}

\begin{lem}
  Let $e$, $B^*$ be as in the lemma, and $t$ be any term. Then 
  \begin{enumerate}
  \item If $e$ is not a subterm of $B(y)$, $\ST{B^*}{e}{t} \equiv
    B(u') \lif B(\meps y{B(y)})$.
  \item If $e$ is a subterm of $B(y)$, i.e., $B(y) \seq B'(e, y)$,
    $\ST{B^*}{e}{t} \equiv B'(t, u') \lif B'(t, \meps{y}{B'(t, y)})$.
  \end{enumerate}
\end{lem}

\begin{proof}
  By inspection.
\end{proof}

\begin{lem}
  If $\proves[\pi]{\ECe} E$ and $E$ does not contain \eps, then there
  is a proof $\pi'$ such that $\proves[\pi']{\ECe} E$ and $\rk{\pi'}
  \le \rk{pi} = r$ and $o(\pi', r) < o(\pi, r)$.
\end{lem}

\begin{proof}
  Let $e$ be an \eps-term critical in $\pi$ and let $A(t_1) \lif
  A(e)$, dots, $A(t_n) \lif A(e)$ be all its critical formulas in
  $\pi$.

  Consider $\ST \pi e t_i$, i.e., $\pi$ with $e$ replaced by $t_i$
  throughout.  Each critical formula belonging to $e$ now is of the
  form $A(t_j') \lif A(t_i)$, since $e$ obviously cannot be a subterm
  of $A(x)$ (if it were, $e$ would be a subterm of $\meps x {A(x)}$,
  i.e., of itself!). Let $\hat\pi_i$ be the sequence of tautologies
  $A(t_i) \lif (A(t_j') \lif A(t_i))$ for $i = 1$, \dots,~$n$,
  followed by $\ST \pi e t_i$.  Each one of the formulas $A(t_j') \lif
  A(t_i)$ follows from one of these by (MP) from $A(t_i)$. Hence,
  $A(t_i) \proves[\hat\pi_i]{\ECe} E$. Let $\pi_i = \hat\pi_i[A_i]$ as
  in Lemma~\ref{ded-lemma}.  We have $\proves[\pi_i]{\ECe} A_i \lif E$.

  The \eps-term $e$ is not critical in $\pi_i$: Its original critical
  formulas are replaced by $A(t_i) \lif (A(t_j') \lif A(t_i))$, which
  are tautologies.  By (1) of the preceding Lemma, no critical
  \eps-term of rank~$r$ was changed at all. By (2) of the preceding
  Lemma, no critical \eps-term of rank $< r$ was replaced by a
  critical \eps-term of rank~$\ge r$. Hence, $o(\pi_i, r) = o(\pi) -
  1$.

  Let $\pi''$ be the sequence of tautologies $\lnot \bigvee_{i=1}^n
  A(t_i) \lif (A(t_i) \lif A(e))$ followed by $\pi$.  Then
  $\bigvee_{i=1}^n A(t_i) \proves[\pi''] E$, $e$ is not critical in
  $\pi''$, and otherwise $\pi$ and $\pi''$ have the same critical
  formulas. The same goes for $\pi''[\lnot \bigvee A(t_i)]$, a proof
  of $\lnot\bigvee A(t_i) \lif E$.  

  We now obtain $\pi'$ as the $\pi_i$, $i = 1$, \dots,~$n$, followed
  by $\pi[\lnot \bigvee_{i=1}^n]$, followed by the tautology 
  \[
  (\lnot \bigvee A(t_i) \lif E) \lif (A(t_1) \lif E) \lif \dots \lif
  (A(t_n) \lif E) \lif E)\dots)
  \]
  from which $E$ follows by $n+1$ applications of (MP).
\end{proof}

\begin{proof}[of the first \eps-Theorem]
  By induction on $o(\pi, r)$, we have: if $\proves[\pi]{\ECe} E$,
  then there is a proof $\pi^*$ of $E$ with $\rk{\pi^-} < r$. By
  induction on $\rk(\pi)$ we have a proof $\pi^{**}$ of $E$ with
  $\rk{\pi^{**}} = 0$, i.e., without critical formulas at all.
\end{proof}

\begin{ex}
  Check these proofs. Can you think of ways to improve the proofs?
\end{ex}

\begin{ex}
  If $E$ contains \eps-terms, the replacement of \eps-terms in the
  construction of $\pi_i$ may change $E$---but of course only the
  \eps-terms appearing as subterms in it.  Use this fact to prove: If
  $\proves{\PCe} E(e)$, then $\proves{\EC} \bigvee_{i=1}^m E(t_j)$ for
  some terms $t_j$.  Can you guarantee that $t_j$ are \eps-free.
\end{ex}

\section[The Case With Identity]{\new{3} The Case With Identity}

In the presence of the identity ($=$) predicate in the language,
things get a bit more complicated. The reason is that instances of the
($=_2$) axiom schema,
\[
t = u \lif (A(t) \lif A(u))
\]
may also contain \eps-terms, and the replacement of an \eps-term~$e$
by a term~$t_i$ in the construction of $\pi_i$ may result in a formula
which no longer is an instance of ($=_2$). For instance, suppose that
$t$ is a subterm of $e = e'(t)$ and $A(t)$ is of the form $A'(e'(t))$.
Then the original axiom is
\[
t = u \lif (A'(e'(t)) \lif A'(e'(u))
\] 
which after replacing $e = e'(t)$ by $t_i$ turns into
\[
t = u \lif (A'(t_i) \lif A'(e'(u)).
\] 
So this must be avoided.  In order to do this, we first observe that
just as in the case of the predicate calculus, the instances of
($=_2$) can be derived from restricted instances. In the case of the
predicate calculus, the restricted axioms are 
\begin{align}
  t = u & \lif (P^n(s_1, \dots, t, \dots s_n) \lif P^n(s_1, \dots, u,
  \dots, s_n) \tag{$=_2'$} \\
  t = u & \lif f^n(s_1, \dots, t, \dots, s_n) = f^n(s_1, \dots, u,
  \dots, s_n) \tag{$=_2''$} \\
\intertext{to which we have to add the \emph{\eps-identity axiom schema:}}
  t = u & \lif \meps x {A(x; s_1, \dots, t, \dots s_n)} = 
  \meps x {A(x; s_1, \dots, u, \dots s_n)} \tag{$=_\eps$}
\end{align}
where $\meps x {A(x; x_1, \dots, x_n)}$ is an \eps-type.

\begin{prop}\label{atomic}
  Every instance of $(=_2)$ can be derived from $(=_2')$, $(=_2'')$,
  and $(=_\eps)$.
\end{prop}

\begin{proof}
  Exercise.
\end{proof}

Now replacing every occurrence of $e$ in an instance of ($=_2'$) or
($=_2''$)---where $e$ obviously can only occur inside one of the terms
$t$, $u$, $s_1$, \dots, $s_n$---results in a (different) instance of
($=_2'$) or ($=_2''$).  The same is true of ($=_\eps$), \emph{provided
  that} the $e$ is neither $\meps x{A(x; s_1, \dots, t, \dots s_n)}$
nor $\meps x{A(x; s_1, \dots, u, \dots s_n)}$.  This would be
guaranteed if the type of $e$ is not $\meps x {A(x; x_1, \dots,
  x_n)}$, in particular, if the rank of $e$ is higher than the rank of
$\meps x {A(x; x_1, \dots, x_n)}$.  Moreover, the result of replacing
$e$ by $t_i$ in any such instance of $(=_\eps$) results in an instance
of $(=_\eps)$ which belongs to the same \eps-type. Thus, in order for
the proof of the first \eps-theorem to work also when $=$ and axioms
$(=_1)$, $(=_2')$, $(=_2''$), and $(=_\eps)$ are present, it suffices
to show that the instances of $(=_\eps)$ with \eps-terms of
rank~$\rk{\pi}$ can be removed.  Call an \eps-term $e$ \emph{special}
in~$\pi$, if $\pi$ contains an occurrence of $t = u \lif e' = e$ as an
instance of $(=_\eps)$.

\begin{thm}
  If $\proves[\pi]{\ECe} E$, then there is a proof $\pi^=$ so that
  $\proves[\pi^=]{\ECe} E$, $\rk{\pi^=} = \rk{pi}$, and the rank of the
  special \eps-terms in $\pi^=$ has rank $< \rk{\pi}$.
\end{thm}

The basic idea is simple: Suppose $t = u \lif e' = e$ is an instance
of $(=_\eps)$, with $e' \seq \meps x {A(x; s_1, \dots, t, \dots s_n)}$
and $e \seq \meps x {A(x; s_1, \dots, u, \dots s_n)}$.  Replace $e$
everywhere in the proof by~$e'$.  Then the instance of $(=_\eps)$
under consideration is removed, since it is now provable from $e' =
e'$.  This potentially interferes with critical formulas belonging to
$e$, but this can also be fixed: we just have to show that by a
judicious choice of~$e$ it can be done in such a way that the other
$(=_\eps)$ axioms are still of the required form.

Let $p = \meps x{A(x; x_1, \dots, x_n)}$ be an \eps-type of rank
$\rk{\pi}$, and let $e_1$, \dots, $e_l$ be all the \eps-terms of
type~$p$ which have a corresponding instance of~$(=_\eps)$ in~$\pi$.
Let $T_i$ be the set of all immediate subterms of $e_1$, \dots, $e_l$,
in the same position as $x_i$, i.e., the smallest set of terms so that
if $e_i \seq \meps x{A(x; t_1, \dots, t_n)}$, then $t_i \in T$. Now
let let $T^*$ be all instances of $p$ with terms from $T_i$
substituted for the $x_i$.  Obviously, $T$ and thus $T^*$ are finite
(up to renaming of bound variables). Pick a strict order $\prec$ on
$T$ which respects degree, i.e., if $\deg{t} < \deg{u}$ then $t \prec
u$.  Extend $\prec$ to $T^*$ by
\[
 \meps x{A(x; t_1, \dots, t_n)} \prec  \meps x{A(x; t'_1, \dots, t'_n)} 
\]
iff 
\begin{enumerate}
\item $\max\{\deg{t_i} : i = 1, \dots, n\} < \max\{\deg{t_i} : i = 1,
  \dots, n\}$ or
\item $\max\{\deg{t_i} : i = 1, \dots, n\} = \max\{\deg{t_i} : i = 1,
  \dots, n\}$  and
\begin{enumerate}
\item $t_i \seq t_i'$ for $i = 1$, \dots,~$k$.
\item $t_{k+1} \prec t_{k+1}'$
\end{enumerate}
\end{enumerate}

\begin{lem}
  Suppose $\proves[\pi]{\ECe} E$, $e$ a special \eps-term in $\pi$
  with $\rk{e} = \rk{\pi}$, $\deg{e}$ maximal among the special
  \eps-terms of rank~$\rk{\pi}$, and $e$ maximal with respect to
  $\prec$ defined above. Let $t = u \lif e' = e$ be an instance of
  $(=_\eps)$ in $\pi$. Then there is a proof $\pi'$,
  $\proves[\pi']{\ECe} E$ such that
  \begin{enumerate}
    \item $\rk{\pi'} = \rk{\pi}$
    \item $\pi'$ does not contain $t = u \lif e' = e$ as an axiom
    \item Every special \eps-term $e''$ of $\pi'$ with the same type
      as $e$ is so that $e'' \prec e$.
    \end{enumerate}
\end{lem}

\begin{proof}
  Let $\pi_0 = \ST \pi e {e'}$. 

  Suppose $t' = u' \lif e''' = e''$ is an $(=_\eps)$ axiom in $\pi$.

  If $\rk{e''} < \rk{e}$, then the replacement of $e$ by $e'$ can only
  change subterms of $e''$ and $e'''$. In this case, the uniform
  replacement results in another instance of $(=_\eps)$ with
  \eps-terms of the same \eps-type, and hence of the same rank $<
  \rk{\pi}$, as the original.

  If $\rk{e''} = \rk{e}$ but has a different type than $e$, then this
  axiom is unchanged in $\pi_0$: Neither $e''$ nor $e'''$ can be $\seq
  e$, bcause they have different \eps-types, and neither $e''$ nor
  $e'''$ (nor $t'$ or $u'$, which are subterms of $e''$, $e'''$) can
  contain $e$ as a subterm, since then $e$ wouldn't be degree-maximal
  among the special \eps-terms of $\pi$ of rank~$\rk{\pi}$.

  If the type of $e''$, $e'''$ is the same as that of $e$, $e$ cannot
  be a proper subterm of $e''$ or $e'''$, since otherwise $e''$ or
  $e'''$ would again be a special \eps-term of rank~$\rk{\pi}$ but of
  higher degree than~$e$.  So either $e \seq e''$ or $e \seq e'''$,
  without loss of generality suppose $e \seq e''$.  Then the
  $(=_\eps)$ axiom in question has the form
  \[
  t' = u' \lif \underbrace{\meps x {A(x; s_1, \dots t', \dots
      s_n)}}_{e'''} = \underbrace{\meps x {A(x; s_1, \dots u', \dots
      s_n)}}_{e'' \seq e}
  \]
  and with $e$ replaced by $e'$:
  \[
  t' = u' \lif \underbrace{\meps x {A(x; s_1, \dots t', \dots
      s_n)}}_{e'''} = \underbrace{\meps x {A(x; s_1, \dots t, \dots
      s_n)}}_{e'}
  \]
  which is no longer an instance of $(=_\eps)$, but can be proved from
  new instances of~$(=_\eps)$.  We have to distinguish two cases
  according to whether the indicated position of $t$ and $t'$ in $e'$,
  $e'''$ is the same or not.  In the first case, $u \seq u'$, and the
  new formula
  \begin{align}
    t' = u & \lif \underbrace{\meps x {A(x; s_1, \dots t', \dots
        s_n)}}_{e'''} = \underbrace{\meps x {A(x; s_1, \dots t, \dots
        s_n)}}_{e'} \notag\\
    \intertext{can be proved from $t = u$ together with} t' = t & \lif
    \underbrace{\meps x {A(x; s_1, \dots t', \dots s_n)}}_{e'''} =
    \underbrace{\meps x {A(x; s_1, \dots t, \dots
        s_n)}}_{e'} \tag{$=_\eps$}\\
    t = u & \lif (t' = u \lif t' = t) \tag{$=_2'$}
  \end{align}
  Since $e'$ and $e'''$ already occured in~$\pi$, by assumption $e'$, $e'''
  \prec e$.

  In the second case, the original formulas read, with terms indicated:
  \begin{align*}
    t = u & \lif \underbrace{\meps x {A(x; s_1, \dots t, \dots, u',
        \dots, s_n)}}_{e'} = \underbrace{\meps x {A(x; s_1, \dots u,
        \dots, u', \dots,
        s_n)}}_{e} \\
    t' = u' & \lif \underbrace{\meps x {A(x; s_1, \dots u, \dots, t',
        \dots, s_n)}}_{e'''} = \underbrace{\meps x {A(x; s_1, \dots u,
        \dots, u', \dots,
        s_n)}}_{e'' \seq e} 
    \intertext{and with $e$ replaced by $e'$ the latter becomes:} 
    t' = u' & \lif \underbrace{\meps x {A(x; s_1, \dots u, \dots, t', \dots
        s_n)}}_{e'''} = \underbrace{\meps x {A(x; s_1, \dots t, \dots,
        u', \dots, s_n)}}_{e'} \\
    \intertext{This new formula is provable from $t = u$ together with}
    u = t & \lif \underbrace{\meps x {A(x; s_1, \dots u, \dots, t', \dots
        s_n)}}_{e'''} = \underbrace{\meps x {A(x; s_1, \dots t, \dots,
        t', \dots, s_n)}}_{e''''} \\
    t' = u' & \lif \underbrace{\meps x {A(x; s_1, \dots t, \dots, t', \dots
        s_n)}}_{e''''} = \underbrace{\meps x {A(x; s_1, \dots t, \dots,
        u', \dots, s_n)}}_{e'}
    \end{align*}
    and some instances of $(=_2')$.  Hence, $\pi'$ contains a
    (possibly new) special \eps-term~$e''''$. However, $e''''
    \prec e$ (Exercise: prove this.)

    In the special case where $e = e''$ and $e' = e'''$, i.e., the
    instance of $(=_\eps)$ we started with, then replacing $e$ by $e'$
    results in $t = u \lif e' = e'$, which is provable from $e' = e'$,
    an instance of $(=_1)$.

    Let $\pi_1$ be $\pi_0$ with the necessary new instances of
    $(=_\eps)$, added. The instances of $(=_\eps)$ in $\pi_1$ satisfy
    the properties required in the statement of the lemma.

    However, the results of replacing $e$ by $e'$ may have impacted
    some of the critical formulas in the original proof.  For a
    critical formula to which $e \seq \meps x {A(x, u)}$ belongs is of the form
    \begin{align}
      A(t', u) & \lif A(\meps x {A(x, u)}, u) \\
      \intertext{which after replacing $e$ by $e'$ becomes}
      A(t'', u) & \lif A(\meps x {A(x, t)}, u) \label{newcrit}
    \end{align}
    which is no longer a critical formula.  This formula, however, can
    be derived from $t = u$ together with 
    \begin{align}
      A(t'', u) & \lif A(\meps x {A(x, t)}, u) \tag{\eps}\\
      t = u & \lif (A(\meps x{A(x, t)}, t) \lif A(\meps x{A(x, t)}, u))
      \tag{$=_2$} \\
      u = t & \lif (A(t'', u) \lif A(t'', t)) \tag{$=_2$}
    \end{align}
    Let $\pi_2$ be $\pi_1$ plus these derivations of (\ref{newcrit})
    with the instances of $(=_2)$ themselves proved from $(=_2')$ and
    $(=_\eps)$.  The rank of the new critical formulas is the same, so
    the rank of $\pi_2$ is the same as that of~$\pi$.  The new
    instances of $(=_\eps)$ required for the derivation of the last
    two formulas only contain \eps-terms of lower rank that that
    of~$e$. (Exercise: verify this.)

    $\pi_2$ is thus a proof of $E$ from $t = u$ which satisfies the
    conditions of the lemma.  From it, we obtain a proof $\pi_2[t =
    u]$ of $t = u \lif E$ by the deduction theorem.  On the other
    hand, the instance $t = u \lif e' = e$ under consideration can
    also be proved trivially from $t \neq u$. The proof $\pi[t \neq
    u]$ thus is also a proof, this time of $t \neq u \lif E$, which
    satisfies the conditions of the lemma. We obtain $\pi'$ by
    combining the two proofs.
\end{proof}

\begin{proof}{Proof of the Theorem}
  By repeated application of the Lemma, every instance of $(=_\eps)$
  involving \eps-terms of a given type~$p$ can be eliminated from the
  proof.  The Theorem follows by induction on the number of different
  types of special \eps-terms of rank~$\rk{\pi}$ in $\pi$.
\end{proof}

\begin{ex} Prove Proposition~\ref{atomic}.
\end{ex}

\begin{ex} Verify that $\prec$ is a strict total order.
\end{ex}

\begin{ex}
  Complete the proof of the Lemma.
\end{ex}

\end{document}